\documentclass{llncs}

\usepackage{psfig}
\usepackage{epsfig}

\newcommand{\R}{{I\!\!R}}

\newcommand{\pa}{\partial}
\newcommand{\be}{\begin{equation}}
\newcommand{\ee}{\end{equation}}
\newcommand{\ba}{\begin{eqnarray}}
\newcommand{\ea}{\end{eqnarray}}
\newcommand{\nn}{\nonumber}
\newcommand{\la}{\label} 
\newcommand{\ep}{h}
\newcommand{\T}{{\cal T}} 
\newcommand{\U}{{\cal U}} 
\newcommand{\dt}{h} 
\newcommand{\e}{{\rm e}} 
\newcommand{\bp}{{\bf p}}
\newcommand{\bq}{{\bf q}}

\newcommand{\bac}{{\bf a}}

\newcommand{\ct}{{\cal T}}
\newcommand{\cu}{{\cal U}}

\newcommand{\bv}{{\bf v}}

\def\R{{\rm I}\! {\rm R}}

\usepackage{amsmath,amssymb}
\usepackage{amsfonts}
\usepackage{makeidx}

\newtheorem{assumption}{Assumption} 

\begin{document}

\pagestyle{headings}

\title{Multi-product operator splitting as a general method of solving
autonomous and non-autonomous equations}
\author{Siu A. Chin$^1$ \and J\"urgen Geiser$^2$}
\institute{Department of Physics, Texas A\&M University, 
College Station, TX 77843, U.S.A. \\
              Tel.: +1-979-845-4190\\
              Fax: +1-979-845-2590\\
              \email{chin@physics.tamu.edu}           
           \and
           Department of Mathematics, Humboldt-Universit\"at zu Berlin,
Unter den Linden 6,
D-10099 Berlin, Germany  \\
              Tel.: +49-30-2093-5451\\
              Fax: +49-30-2093-5859\\
              \email{geiser@mathematik.hu-berlin.de} 
}

\maketitle

\begin{abstract}
Prior to the recent development of symplectic integrators,
the time-stepping operator $\e^{h(A+B)}$ was routinely decomposed 
into a sum of products of $\e^{h A}$ and $\e^{hB}$ in the
study of hyperbolic partial differential equations. In the 
context of solving Hamiltonian dynamics,
we show that such a decomposition give rises to both {\it even}
and {\it odd} order Runge-Kutta and Nystr\"om 
integrators. By use of Suzuki's forward-time derivative
operator to enforce the time-ordered exponential, 
we show that the same decomposition can be used to solve
non-autonomous equations. In particular, odd order algorithms
are derived on the basis of a highly non-trivial {\it time-asymmetric} kernel.
Such an operator approach
provides a general and unified basis for understanding
structure non-preserving algorithms and is especially useful
in deriving very high-order algorithms via {\it analytical} extrapolations. 
In this work, algorithms up to the 100th order are tested by 
integrating the ground state wave function of the hydrogen atom.
For such a singular Coulomb problem, the multi-product expansion 
showed uniform convergence and is free of poles
usually associated with structure-preserving methods. 
Other examples are also discussed. 

\end{abstract}


\section{Introduction}

 In the course of devising numerical algorithms for solving
 the prototype linear hyperbolic equation

\be
\partial_t u = A u_x+B u_y, \qquad u(0) = u_0 ,
\la{equ1}
\ee
where $A$ and $B$ are non-commuting matrices,
Strang\cite{strang68} proposed two second-order algorithms corresponding 
to approximating
\be
{\cal T}(\ep)=\e^{\ep(A+B)}
\la{expab}
\ee
either as
\be
S(\ep)=\frac12\left( \e^{\ep A}\e^{\ep B}+\e^{\ep B}\e^{\ep A}\right)
\la{str1}
\ee
or as
\be
S_{AB}(\ep)=\e^{(\ep/2) B}\e^{\ep A} \e^{(\ep/2) B}.
\la{str2}
\ee
Following up on Strang's work, Burstein and Mirin\cite{bur70} suggested
that Strang's approximations can be generalized to higher orders
in the form of a multi-product expansion (MPE),
\be
{\rm e}^{\ep( A+ B)}=\sum_k c_k
\prod_{i}
{\rm e}^{a_{ki}\ep A}{\rm e}^{b_{ki}\ep B}
\label{mprod} 
\ee
and gave two third-order approximations
\be
D(\ep)=\frac43 \left(\frac{S_{AB}(\ep)+S_{BA}(\ep)}{2}\right)-\frac13 S(\ep)
\la{dun}
\ee
and
\be
B_{AB}(\ep)=\frac98 \e^{(\ep/3) A}\e^{(2\ep/3) B}\e^{(2\ep/3) A}\e^{(\ep/3) B}
-\frac18 \e^{\ep A}\e^{\ep B}.
\la{bmsch}
\ee 
They credited J. Dunn for finding the decomposition $D(\ep)$ and
noted that the weights ${c_k}$ are no longer positive beyond 
second order. Thus the  stability of the entire algorithm can no longer be 
inferred from the stability of each component product.

Since (\ref{str1}), (\ref{str2}), (\ref{dun}) and (\ref{bmsch}) are
approximations for the exponential of two general operators, 
they can be applied to problems unrelated to solving 
hyperbolic partial differential equations. For example,
the evolution of any dynamical variable $u(\bq,\bp)$ (including $\bq$ and $\bp$ themselves)
is given by the
Poisson bracket,
\begin{equation}
\pa_tu(\bq,\bp)=
                 \Bigl(
		          {{\partial u}\over{\partial \bq}}\cdot
                  {{\partial H}\over{\partial \bp}}
				 -{{\partial u}\over{\partial \bp}}\cdot
                  {{\partial H}\over{\partial \bq}}
				                    \Bigr)=(A+B)u(\bq,\bp).
\label{peq}
\end{equation}
For a separable Hamiltonian, 
\begin{equation}
H(\bp,\bq)={\bp^2\over{2m}}+V(\bq),
\label{ham}
\end{equation}
$A$ and $B$ are Lie operators, or vector fields
\be
A=\bv\cdot\frac{\pa}{\pa\bq} \qquad B=\bac(\bq)\cdot\frac{\pa}{\pa\bv}
\la{shop} 
\ee
where we have abbreviated $\bv=\bp/m$ and $\bac(\bq)=-\nabla V(\bq)/m$.
The exponential operators $\e^{h A}$ and $\e^{h B}$ are then just shift operators,
with $S(h)$ giving the second-order Runge-Kutta integrator
\ba
\bq&=&\bq_0+h\bv_0+\frac12 h^2\bac(\bq_0)\equiv\bq_1\\
\bv&=&\bv_0+\frac{h}2\Bigl[\bac(\bq_0)+\bac(\bq_0+h\bv_0)\Bigr]
\ea
and $S_{AB}(h)$, the symplectic Verlet or leap-frog algorithm
\ba
\bq&=&\bq_1\la{lfq}\\
\bv&=&\bv_0+\frac{h}2\Bigl[\bac(\bq_0)+\bac(\bq)\Bigr].
\la{lfv}
\ea
More interestingly, Dunn's decomposition $D(h)$ gives
\ba
\bq&=&\bq_0+h\bv_0+\frac{h^2}6 \Bigl[\bac(\bq_0)+2\bac(\bq_0+\frac{h}2\bv_0)\Bigr]
\la{kutta3q}\\
\bv&=&\bv_0+\frac{h}6\Bigl[\bac(\bq_0)+4\bac(\bq_0+\frac{h}2\bv_0)
+2\bac(\bq_1)-\bac\Bigl(\bq_1-\frac12 h^2\bac(\bq_0)\Bigr)
\Bigr].
\la{kutta3}
\ea
Since
\be
2\bac(\bq_1)-\bac\Bigl(\bq_1-\frac12 h^2\bac(\bq_0)\Bigr)
=\bac\Bigl(\bq_1+\frac12 h^2\bac(\bq_0)\Bigr)+O(h^4),
\la{sub1}
\ee
it remains correct to third order to write
\be
\bv=\bv_0+\frac{h}6\Bigl[\bac(\bq_0)+4\bac(\bq_0+\frac{h}2\bv_0)
+\bac\Bigl(\bq_0+h\bv_0+h^2\bac(\bq_0)\Bigr)
\Bigr].
\la{kutta3v}
\ee
One recognizes that (\ref{kutta3q}) and (\ref{kutta3v}) as 
{\it precisely} Kutta's third 
order algorithm\cite{hil74} for solving a second-order differential equation.
Burstein and Mirin's approximation $B_{AB}(h)$ directly gives, without
any change, 
\ba
\bq&=&\bq_0+h\bv_0+\frac{h^2}4 \Bigl[\bac(\bq_0)
+\bac(\bq_{2/3})\Bigr]
\la{heun3q}\\
\bv&=&\bv_0+\frac{h}4\Bigl[\bac(\bq_0)+3\bac(\bq_{2/3})\Bigr],
\la{ab3}
\ea
with 
\be
\bq_{2/3}\equiv\bq_0+\frac23h\bv_0+\frac29 h^2\bac(\bq_0),
\la{q23}
\ee
which is Nystr\"om's third order algorithm requiring only two force-evaluations\cite{bat99,nys25}.
Since Burstein and Mirin's approximation is not symmetric,
$B_{BA}(h)$ produces a different algorithm
\ba
\bq&=&\bq_0+h\bv_0+\frac{h^2}2\bac_{1/3}
\la{ba3q}\\
\bv&=&\bv_0+\frac{h}4\Bigl[ 3\bac_{1/3}
+\frac32\bac(\bq_0+h\bv_0+\frac49 h^2\bac_{1/3})-\frac12\bac(\bq_0+h\bv_0)
\Bigr],
\la{ba3}
\ea
where $\bac_{1/3}=\bac(\bq_0+h\bv_0/3)$. Again, since
\be
\frac32\bac(\bq_0+h\bv_0+\frac49 h^2\bac_{1/3})-\frac12\bac(\bq_0+h\bv_0)
=\bac(\bq_0+h\bv_0+\frac23 h^2\bac_{1/3}) +O(h^4),
\la{sub2}
\ee
(\ref{ba3}) can be rewritten as
\be
\bv=\bv_0+\frac{h}4\Bigl[ 3\bac_{1/3}+\bac(\bq_0+h\bv_0+\frac23 h^2\bac_{1/3})
\Bigr].
\la{ba3f}
\ee
Eqs.(\ref{ba3q}) and (\ref{ba3f}) is a new third order
algorithm with two force-evaluations but without evaluating 
the force at the starting position.  
More recently, Ref.\cite{chin08} has shown that Nystr\"om's four-order
algorithm\cite{bat99} with three force-evaluations and Albrecht's six-order 
algorithm\cite{alb55} with five-force evaluations can all be derived from 
operator expansions of the form (\ref{mprod}). 

Just as symplectic integrators\cite{neri87,fr90} can be derived from a 
{\it single product} splitting,
\be
{\rm e}^{\ep( A+ B)}=
\prod_{i}
{\rm e}^{a_{i}\ep A}{\rm e}^{b_{i}\ep B},
\label{sprod} 
\ee
these examples clearly show that the 
{\it multi-product} splitting (\ref{mprod}) is the fundamental 
basis for deriving non-symplectic, Nystr\"om type algorithms.
(These are not fully Runge-Kutta algorithms, because the operator $B$ in (\ref{shop})
would not be a simple shift operator if $\bac(\bq)$ becomes dependent on $\bv$.
On the other hand, Nystr\"om type algorithms are all that are necessary for the study of 
most Hamiltonian systems.) As illustrated above, one goal of this work is to show that 
all traditional results on Nystr\"om integrators can be much more simply 
derived and understood on the basis of multi-product splitting. In fact, we have the
following theorem 
\begin{theorem}
Every decomposition of ${\rm e}^{\ep( A+ B)}$ in the form of
\be
\sum_k c_k
\prod_{i}
{\rm e}^{a_{ki}\ep A}{\rm e}^{b_{ki}\ep B}={\rm e}^{\ep( A+ B)}+O(\ep^{n+1}),
\label{mprodth} 
\ee
where $A$ and $B$ are non-commmuting operators, with real coefficients 
$\{c_k, a_{ki}, b_{ki}\}$ and finite indices $k$ and $i$, 
produces a nth-order Nystr\"om integrator.
\end{theorem}
(Note that the order $n$ of the integrator is defined with respect to
the error in approximating the operator $(A+B)$ and therefore the error in the
time-stepping operator is one order higher.)
The resulting integrator, however, may not be optimal. As
illustrated above, at low orders, some force evaluations can be
combined without affecting the order of the integrator. However, such a
force consolidation is increasely unlikely at higher orders. 
This theorem produces, both the traditional Nystr\"om integrators where the force 
is always evaluated initially, and non-FASL(First as Last) 
integrators where the force is never evaluated initially, as in
(\ref{ba3q}) and (\ref{ba3f}). 

The advantage of a single product splitting is that the resulting 
algorithms are structure-preserving, such as being  
symplectic, unitary, or remain within the group manifold. However, 
single product splittings beyond the second-order requires exponentially 
growing number of operators with unavoidable negative coefficients\cite{sheng89,suzuki91} 
and cannot be applied to time-irreversible or semi-group problems. 
Even for time-reversible systems where negative time steps are not a problem, 
the exponential growth on the number of force evaluations renders high order symplectic 
integrators difficult to derived and expensive to use. For example, it has been found 
empirically that symplectic algorithms of orders 4, 6, 8 and 10, required a minimum of 
3, 7, 15 and 31 force-evaluations respectively\cite{chin08}. Here, we show that
{\it analytically} extrapolated
algorithms of {\it odd} orders 3, 5, 7, 9
only requires 2, 4, 7, 11 force-evaluations and algorithms of
{\it even} orders 4, 6, 8, 10, only
require 3, 5, 10, 15 force evaluations. Thus at the tenth order,
an extrapolated MPE integrator, only requires half the
computational effort of a symplectic integrator. Or, for 28 force-evaluations,
one can use a 14th order MPE integrator instead.
This is a great advantage in many practical calculations where long term accuracy 
and structure preserving is not an issue.
The advantage is greater still beyond the tenth order, where no symplectic
integrators and very few RKN algorithms are known. 
Here, we demonstrated the working of MPE algorithms up to the 100th order.

By use of Suzuki\cite{suzu93} method of implementing the
time-ordered exponential, this work shows that 
the multi-product expansion (\ref{mprodth}) can
be easily adopted to solve the non-autonomous equation
\be
\partial_t Y(t) = A(t)Y(t), \quad Y(0) = Y_0.
\la{funeq}
\ee
In even-order cases, this method reproduces Gragg's\cite{gragg65}
classical result in just a few lines. In odd-order cases, this
method demonstrates a highly non-trivial extrapolation of 
a {\it time-asymmetric} kernel, which has never been anticipated before.
Finally, we show that the multi-product expansion (\ref{mprodth}) 
converges uniformly, in contrast to structure-preserving methods, 
such as the Magnus expansion, 
which generally has a finite radius of convergence. The convergence
of (\ref{mprodth}) is verified in various analytical and numerical examples, 
up to the 100th order.

The paper is outlined as follows:
In Section \ref{mpedec}, we derive key results of MPE, including the
extrapolation of odd-order algorithms. 
In Section \ref{suzuki}, we show how Suzuki's method can be used
to transcribe any splitting scheme for solving non-autonomous equations.
In Section \ref{error}, we present an error and convergence analysis of
the multi-product splitting based on extrapolation.
Numerical examples and comparison to the Magnus expansion are given in 
Section \ref{exp}.
In Section \ref{conc}, we briefly summarize our results. 

\section{Multi-product decomposition}
\label{mpedec}

The multi-product decomposition (\ref{mprod}) is obviously more complicated 
than the single product splitting (\ref{sprod}). Fortunately, 
nineteen years after Burstein and Mirin, Sheng\cite{sheng89} proved their
observation that beyond second-order, $a_{ki}$, $b_{ki}$ and $c_k$ cannot all be
positive. This negative result, surprisingly, can be used to completely determine 
$a_{ki}$, $b_{ki}$ and $c_k$ to all orders. This is because   
for general applications, including solving time-irreversible problems, 
one must have $a_{ki}$ and $b_{ki}$ positive.
Therefore every single product in (\ref{mprod}) can at most be 
second-order\cite{sheng89,suzuki91}. But such a product is
easy to construct, because every left-right symmetric single product {\it is} 
second-order. Let ${\cal T}_S(h)$ be such a product with
$\sum_i a_{ki}=1$ and $\sum_i b_{ki}=1$, then
${\cal T}_S(h)$ is time-symmetric by construction, 
\be
{\cal T}_S(-h){\cal T}_S(h)=1,
\ee
implying that it has only odd powers of $h$
\be
{\cal T}_S(\ep)=\exp(\ep (A+B)+\ep^3 E_3+\ep^5 E_5+\cdots)
\la{secerr}
\ee
and therefore correct to second-order.
(The error terms $E_i$ are nested commutators of $A$ and $B$
depending on the specific form of $\ct_S$.)
This immediately suggests that the $k$th power of ${\cal T}_S$ at 
step size $h/k$ must have the form 
\be
{\cal T}_S^k(\ep/k) 
=\exp(\ep (A+B)+k^{-2}\ep^3 E_3+k^{-4}\ep^5 E_5+\cdots),
\la{thk}
\ee
and can serve as a basis for the multi-production expansion (\ref{mprod}).
The simplest such symmetric product is  
\be
{\cal T}_2(h)=S_{AB}(h) \quad{\rm or}\quad {\cal T}_2(h)=S_{BA}(h).
\ee
If one naively assumes that
\be
{\cal T}_2(h)=\e^{\ep(A+B)}+ Ch^3+Dh^4+\cdots,
\ee
then a Richardson extrapolation would only give
\be
\frac1{k^2-1}\left[k^2 \ct_2^k(h/k)-\ct_2(h)\right]=\e^{\ep(A+B)}+ O(h^4),
\ee
a third-order\cite{schat94} algorithm. However, because the error structure
of $\ct_2(h/k)$ is actually given by (\ref{thk}), one has
\be
{\cal T}_2^k(h/k)=\e^{\ep(A+B)}+ k^{-2}h^3E_3+\frac12 k^{-2}h^4[(A+B)E_3+E_3(A+B)]+O(h^5),
\ee
and {\it both} the third and fourth order errors can be eliminated simultaneously, 
yielding a fourth-order algorithm. Similarly, the leading $2n+1$ and $2n+2$ order 
errors are multiplied by $k^{-2n}$ and can be eliminated at the same time. Thus 
for a given set of $n$ whole numbers $\{k_i\}$ one can have a $2n$th-order approximation  
\be
{\rm e}^{\ep(A+B)}
=\sum_{i=1}^n c_i{\cal T}_2^{k_i}\left(\frac\ep{k_i}\right)
+O(h^{2n+1}).
\la{mulexp}
\ee
provided that $c_i$ satisfy the simple Vandermonde equation: 
\be
\left(
\begin{array}{c c c c c}
1 & 1 & 1 & \ldots & 1 \\
k_1^{-2} &  k_2^{-2} & k_3^{-2} & \ldots & k_n^{-2} \\
k_1^{-4} &  k_2^{-4} & k_3^{-4} & \ldots & k_n^{-4} \\
\ldots & \ldots & \ldots & \ldots & \ldots \\
k_1^{-2 (n - 1)} &  k_2^{-2 (n - 1)} & k_3^{-2 (n - 1)} & \ldots & k_n^{-2 (n - 1)}
\end{array}
\right)
\left(
\begin{array}{c}
c_1 \\
c_2 \\
c_3 \\
\ldots \\
c_n
\end{array}
\right)
=
\left(
\begin{array}{c}
1 \\
0 \\
0 \\
\ldots \\
0
\end{array}
\right)
\la{vand}
\ee
Surprisely, this equation has closed form solutions\cite{chin08} for all $n$
\be
c_i=\prod_{j=1 (\ne i)}^n\frac{k_i^2}{k_i^2-k_j^2}.
\la{coef}
\ee
The natural sequence $\{k_i\}=\{1, 2, 3 \,...\, n\}$ produces a $2n$th-order
algorithm with the minimum $n(n+1)/2$ evaluations of $\ct_2(h)$.
For orders four to ten, one has explicitly:
\be
{\cal T}_4(\ep)=-\frac13{\cal T}_2(\ep)
+\frac43{\cal T}_2^2\left(\frac\ep{2}\right)
\la{four}
\ee
\be
{\cal T}_6(\ep)=\frac1{24} {\cal T}_2(\ep)
-\frac{16}{15}{\cal T}_2^2\left(\frac\ep{2}\right)
+\frac{81}{40}{\cal T}_2^3\left(\frac\ep{3}\right)
\la{six}
\ee
\be
{\cal T}_8(\ep)=-\frac1{360} {\cal T}_2(\ep)
+\frac{16}{45}{\cal T}_2^2\left(\frac\ep{2}\right)
-\frac{729}{280}{\cal T}_2^3\left(\frac\ep{3}\right)
+\frac{1024}{315}{\cal T}_2^4\left(\frac\ep{4}\right)
\la{eight}
\ee
\ba
&&{\cal T}_{10}(\ep)=\frac1{8640} {\cal T}_2(\ep)
-\frac{64}{945}{\cal T}_2^2\left(\frac\ep{2}\right)
+\frac{6561}{4480}{\cal T}_2^3\left(\frac\ep{3}\right)\nn\\
&&\qquad\qquad\quad-\frac{16384}{2835}{\cal T}_2^4\left(\frac\ep{4}\right)
+\frac{390625}{72576}{\cal T}_2^5\left(\frac\ep{5}\right).
\la{ten}.
\ea
As shown in Ref.\cite{chin08}, $\ct_4(h)$ reproduces Nystr\"om's fourth-order 
algorithm with three force-evaluations
and $\ct_6(h)$ yielded a new
sixth-order Nystr\"om type algorithm with five force-evaluations.

\begin{remark}
It is easy to show that the Verlet algorithm (\ref{lfq}) and (\ref{lfv}) corresponding
to $S_{AB}(h)$ produces the same trajectory as St\"omer's second order scheme
\be
\bq_1-2\bq_0+\bq_{-1}=h^2\bac(\bq_0).
\la{strom}
\ee
However, it is extremely difficult to deduce from (\ref{strom}) that the underlying
error structure is basically (\ref{secerr}) and allows for a $h^2$-extrapolation.
This is the great achievment of Gragg\cite{gragg65}. Nevertheless, the power of the present 
operator approach is that we can reproduce his results in a few lines. 
The error structure here, (\ref{secerr}), is a simple consequence of the symmetric
character of the product, and allows us to bypass Gragg's lengthy proof on the asymptotic 
errors of (\ref{strom}). Moreover, this $h^2$-extrapolation can be applied to any $\ct_S(h)$,
not necessarily restricted to (\ref{strom}). For example, the use of $S_{BA}(h)$
produces an entirely different sequence of extrapolations\cite{chin08}, distinct from 
from that based on (\ref{strom}).  
\end{remark}

\begin{remark}
In the original work of Gragg, the use of (\ref{strom}) as the basis for his
extrapolation is a matter of default; it is a well-known second-order solution.
Here, in extrapolating operators, the use of $S_{AB}(h)$ or $S_{BA}(h)$ is for the 
specific purpose that they can be applied to time-irreversible problems. 
While all positive time steps algorithms are possible in the fourth-order\cite{suzu96,chin97}
by including the operator $[B,[A,B]]$, MPE is currently the only way of producing 
sixth and higher-order algorithms in solving the imaginary time Schr\"odinger 
equation\cite{chin093} and in doing Path-Integral Monte Carlo simulations\cite{zill10}.  
The fact that MPE is no longer norm preserving nor even strictly positive, 
does not affect the higher order convergences in these applications. These
non-structure preserving elements are within the error noise of the algorithm.
MPE is less useful in solving the real time Schr\"oding equation where, unitarity 
is of critical importance.
 
\end{remark}

\begin{remark} The explicit coefficient $c_i$ coincide with the diagonal elements of the 
Richardson-Aitken-Neville extrapolation\cite{hair93} table. This is not surprising, since
they are coefficients of extrapolation. As shown in \cite{chin08}, $c_i=L_i(0)$, where
$L_i(x)$ are the Lagrange interpolating polynomials with interpolation points $x_i=k^{-2}_i$. 
What is novel here is that $c_i$ is known analytically 
and a simple routine calling it repeatedly to execute $\ct_2(h)$ will generate an
arbitrary even order algorithm without any table construction. The resulting algorithm
is extremely portable and compact and can serve as a benchmark by which all 
integrators of the same order can be compared. In Ref.\cite{chin08},  
the only algorithm that have outperformed MPE is Dormand and Prince's\cite{dorm87} 
12th-order integrator as given in Ref.\cite{brank89}.
\end{remark}

Having the explicit solutions $c_i$ now suggests new ways of solving old problems.
Suppose one wishes to integrate the system to time $t$. One may begin by using 
a second-order algorthm and iterate it $m$ time at time step $h=t/m$,
\be
\T_{2,m}(h)=\T_2^m(t/m).
\la{t2n}
\ee
Every position on the trajectory will then be correct to second order in $h$.
However, if one were only interested in the final position at time $t$, then 
one can correct this {\it final position} to fourth order by simply computing
{\it one} more $\T_2(t)$ and modify (\ref{t2n}) via
\be
\T_{4,m}(h)=\frac{m^2}{m^2-1^2}\T_2^m(t/m)-\frac{1^2}{m^2-1^2}\T_2(t),
\la{t4n}
\ee 
or correct it to sixth-order via
\be
\T_{6,m}(h)=\frac{m^4\T_2^m(t/m)}{(m^2-1^2)(m^2-2^2)}
+\frac{2^4\T_2^2(t/2)}{(2^2-1^2)(2^2-m^2)}
+\frac{1^4\T_2(t)}{(1^2-2^2)(1^2-m^2)},
\ee 
and so on, to any even order. The expansion coefficients are given by
$\{k_i\}$ equal to 
$\{m,1\}$, $\{m,2,1\}$, $\{m,3,2,1\}$ {\it etc.}. This is similar to the idea of 
process algorithms\cite{blan99}, but much, much simpler. The processor for
correcting (\ref{t2n}) beyond the fourth-order can be quite complex if the
entire algorithm were to remain symplectic. Here, for Nystr\"om integrators,
the extrapolation coefficient is known to all even orders. 
Alternatively, one can view the above as correcting every $m$th step of 
the basic algorithm $\T_2(t/m)$ over a short time interval of $t$. 
Thus knowing $c_i$ allows great flexibility is designing algorithms that run
the gamut from being correct to arbitrary high order at every time-step, every other
time-step, every third time-step, {\it etc.}, to only at the final time step. With
MPE, one can easily produce versatile adaptive algorithms by varying both the time step 
size $h$ and the order of the algorithm.

\begin{remark}
Since MPE is an extrapolation, it is expected to be more proned to
round-off errors. Thus if $n$ is too large in (\ref{t4n}), the second term
maybe too small and the correction is lost to round-off errors. However, as seen in 
(\ref{sub1}) and (\ref{sub2}), the required substractions are sometime 
well-defined and the the round-off errors are within the error noise of the algorithm. 
As will be shown in Section \ref{exp},
the round-off errors are sometime less severe than expected.

\end{remark}

\begin{remark} The idea of extrapolating symplectic algorithms has been considered by 
previously by Blanes, Casas and Ros\cite{bcr99} and Chan and Murua\cite{cm00}. 
They studied the case of extrapolating an $2n$-order symplectic integrator. 
They did not obtain analytical forms for their expansion coefficients but noted that 
extrapolating a $2n$-order symplectic integrator will preserve the symplectic character 
of the algorithm to order $4n+1$. While this is more general, such an extrapolation 
cannot be applied to time-irreversible systems for $n>1$. 
\end{remark}

Finally, we note that
\be
{\rm e}^{\ep( A+ B)}=\lim_{n\rightarrow\infty}\sum_{i=1}^n c_i {\cal T}_2^{k_i}
\left(\frac\ep{k_i}\right).
\la{mcomp}
\ee
In principle, for any countable sets of $\{k_i\}$, we have achieved an exact decomposition, 
with known coefficients. This is in contrast to the structure-preserving, but
impractical Zassenhaus formula.

The above derivation of even-order algorithms, is at most an elaboration on
Gragg's seminal work. Below, we will derive arbitrary {\it odd-order} Nystr\"om 
algorithms which have not been anticipated in any classical study. Since
\be
\ct_1(h)=\e^{hA}\e^{hB}=\exp[\ep (A+B)+\ep^2 F_2+\ep^3 F_3+\ep^4 F_4+\cdots],
\la{firsterr}
\ee
contain errors of all orders ($\{F_i\}$ are nested commutators of the usual 
Baker-Campbell-Hausdorff formula),
extrapolations based on $\ct_1^k(h/k)$ will not yield a $h^2$-order scheme. 
However, there is a $h^2$-order basis
hidden in Burstein and Mirin's original decomposition (\ref{bmsch}). 
The following basis for $n=1,2,3 \dots$
\be
{\cal U}_n(h)=\e^{\frac{h}{2n-1}A}
(\e^{\frac{2h}{2n-1}B}\e^{\frac{2h}{2n-1}A})^{n-1}\e^{\frac{h}{2n-1}B}
\la{ubasis}
\ee
has the remarkable property that it {\it effectively} behaves as if
\ba
\label{struc_1}
{\cal U}_n(h)=\exp[\ep (A+B)+x^{-2}(\ep^2 F_2+\ep^3 F_3)+x^{-4}(\ep^4 F_4+\ep^5 F_5)+\cdots]
\la{eff}
\ea
where $x=(2n-1)$.
(By effectively we mean that ${\cal U}_n(h)$ actually has the form
\ba
{\cal U}_n(h)&=&\exp[\ep (A+B)+x^{-2}(\ep^2 F_2+\ep^3 F_3)    \nn\\
&&\qquad +(x^{-2}-x^{-4})h^4F^\prime_4 +x^{-4}(\ep^4 F_4+\ep^5 F_5)+\cdots]
\la{act}
\ea
where $F^\prime_4$ are additional commutators not present in (\ref{firsterr}).
However, this is essentially (\ref{eff}) with altered $F_i$ but without 
changing the crucial power pattern of $x^{-2k}$.) In this case,
(\ref{eff}) (as well as (\ref{act})) can be extrapolated similarly as in the 
even order case,  
\be
{\rm e}^{\ep(A+B)}
=\sum_{i=1}^n \tilde{c}_i{\cal U}_{i} (\ep) + O(h^{2n}),
\la{mulexp_2}
\ee
where $\tilde{c}_i$ satisfies the same 
Vandermonde equation (\ref{vand}), with the same solution (\ref{coef}),
but with $\{k_i\}$ consists of only {\it odd} whole numbers. The first few odd
order decompositions corresponding to $\{k_i\}$ being
$\{1,3\}$, $\{1,3,5\}$, $\{1,3,5,7\}$ and $\{1,3,5,7,9\}$ are:
\be
\ct_3(\ep)=-\frac18\cu_1(\ep)+\frac98\cu_2(\ep)
\la{three}
\ee
\be
\ct_5(\ep)=\frac1{192}\cu_1(\ep)-\frac{81}{128}\cu_2(\ep)+\frac{625}{384}\cu_3(\ep)
\la{five}
\ee
\be
\ct_7(\ep)=-\frac1{9216}\cu_1(\ep)+\frac{729}{5120}\cu_2(\ep)
-\frac{15625}{9216}\cu_3(\ep)+\frac{117649}{46080}\cu_4(\ep)
\la{seven}
\ee
\ba
\ct_9(\ep)&=&\frac1{737280}\cu_1(\ep)-\frac{729}{40960}\cu_2(\ep)
+\frac{390625}{516096}\cu_3(\ep)\nn\\
&&\qquad  -\frac{5764801}{1474560}\cu_4(\ep)+\frac{4782969}{1146880}\cu_5(\ep)
\la{nine}.
\ea

The splitting $\ct_3(h)$ explains the original form of Burstein and Mirin's decomposition
and Nystr\"om's third-order algorithm. The splitting $\ct_5(h)$
again produces, without any tinkering, Nystr\"om's fifth-order 
integrators\cite{bat99} with four force-evaluations:   
\ba
\bq&=&\bq_0+h\bv_0+\frac{h^2}{192}
\Bigl[23 \bac_0 +75\bac_{2/5}-27\bac_{2/3}+25\bac_{4/5}\Bigr]
\la{alg5q}\\
\bv&=&\bv_0+\frac{h}{192}\Bigl[23\bac_0+125\bac_{2/5}-81\bac_{2/3}+125\bac_{4/5}  \Bigr],
\la{alg5v}
\ea
where we have denoted $\bac_{i/k}=\bac(\bq_{i/k})$ with
\ba
\bq_{2/5}&=&\bq_0+\frac25 h\bv_0+\frac2{25}h^2\bac_0 \nn\\
\bq_{4/5}&=&\bq_0+\frac45 h\bv_0+\frac4{25}h^2(\bac_0+\bac_{2/5})
\la{fiveq}
\ea
and where $\bq_{2/3}$ has been given earlier in (\ref{q23}).
(Interchange of $A\leftrightarrow B$ in $\ct_5(h)$ will also yield
a fifth-order algorithm, but since the final force-evaluations can only
be combined as in (\ref{sub2}) to order $O(h^4)$, such a force consolidation 
cannot be used for a fifth-order algorithm. The algorithm will then
require six force-evaluations, which is undesirable. We shall therefore
ignore this alternative case from now on.)
With three more force-evaluations at
\ba
\bq_{2/7}&=&\bq_0+\frac27 h\bv_0+\frac2{49}h^2\bac_0 \nn\\
\bq_{4/7}&=&\bq_0+\frac47 h\bv_0+\frac4{49}h^2(\bac_0+\bac_{2/7}) \nn\\
\bq_{6/7}&=&\bq_0+\frac67 h\bv_0+\frac2{49}h^2(3\bac_0+4\bac_{2/7}+2\bac_{4/7}),
\la{sevenq}
\ea
$\ct_7(h)$ produces the
following seventh-order algorithm with seven force-evaluations, which has never been 
derived before,
\ba
\bq&=&\bq_0+h\bv_0+\frac{h^2}{23040}
\Bigl[
1682\bac_0+729\bac_{2/3} -3125(3\bac_{2/5}+\bac_{4/5})\nn\\
&&\qquad\qquad\qquad\qquad\qquad\qquad\qquad +2401(5\bac_{2/7}+3\bac_{4/7}+\bac_{6/7})
\Bigr]
\la{alg7q}\\
\bv&=&\bv_0+\frac{h}{23040}\Bigl[
1682\bac_0+2167\bac_{2/3} -15625(\bac_{2/5}+\bac_{4/5})\nn\\
&&\qquad\qquad\qquad\qquad\qquad\qquad\qquad +16807(\bac_{2/7}+\bac_{4/7}+\bac_{6/7})
\Bigr].
\la{alg7v}
\ea
These analytical derivations are of course unnecessary in practical applications. As
in the even-order case, both the coefficients $c_k$ and the algorithm corresponding to
$\cu_{n}(h)$ can be called repeatedly to generate any odd-order integrators. 
   
Since each $\cu_n(h)$ requires $n$ force evaluation, but have the initial force
in common, each $(2n-1)$ order algorithm requires $\frac12 n(n-1)+1$
force-evaluations. Thus for odd-orders 3, 5, 7, 9, the number of force-evaluation
required are 2, 4, 7, 11. As alluded to earlier, for even-order 4, 6, 8, 10,
the number of force-evaluation required are 3, 5, 10, 15. These   
sequences of extrapolated algorithms therefore provide 
a natural explanation for the order barrier in Nystr\"om algorithms. 
For order $p<7$, the number of force-evaluation can be $p-1$, but for
$p>7$, the number of force-evaluation must be greater than $p$.

\begin{remark}
In general we have the following order notation for the even and odd algorithms:
\begin{itemize}
\item The order of the even
algorithm is $2n$, its decomposition error is $2n+1$. 
\item The order of the
odd algorithm is $2n-1$, its decomposition error is $2n$. 
\end{itemize}
\end{remark}

\section{Solving non-autonomous equations}
\label{suzuki}

The solution to the non-autonomous equation (\ref{funeq}) can
be formally written as  

\begin{equation}
Y(t+\dt)=\T\Bigl(\exp\int_t^{t+\dt}A(s)ds\Bigr)Y(t),
\label{expth}
\end{equation}
aside from the conventional expansion
\begin{equation}
\T\Bigl(\exp\int_t^{t+\dt}A(s)ds\Bigr)=1+\int_t^{t+\dt}A(s_1)ds_1
+\int_t^{t+\dt}ds_1\int_t^{s_1}ds_2A(s_1)A(s_2)+\cdots,
\label{texpexp}
\end{equation}
the time-ordered exponential can also be 
interpreted more intuitively as
\begin{eqnarray}
\T\Bigl(\exp\int_t^{t+\dt}A(s)ds\Bigr)
=&&\lim_{n\rightarrow\infty}
\T\Bigl({\rm e}^{
{{\dt}\over n}
\sum_{i=1}^{n}A(t+i{{\dt}\over n})}\Bigr),\la{tor1}\\
=&&\lim_{n\rightarrow\infty}
{\rm e}^{ {{\dt}\over n}A(t+\dt)}
\cdots
{\rm e}^{ {{\dt}\over n}A(t+{{2\dt}\over n})}
{\rm e}^{ {{\dt}\over n}A(t+{{\dt}\over n})}.
\label{tor2}
\end{eqnarray}
The time-ordering is trivially accomplished in going
from (\ref{tor1}) to (\ref{tor2}). To enforce latter,
Suzuki\cite{suzu93} introduces the {\it forward time derivative} operator, also called super-operator:
\begin{equation}
D={{\buildrel \leftarrow\over\partial}\over{\partial t}}
\label{ftsh}
\end{equation}
such that for any two time-dependent functions $F(t)$ and $G(t)$,
\begin{equation}
F(t){\rm e}^{\dt D}G(t)=F(t+\dt)G(t).
\label{fg}
\end{equation}
If $F(t) = 1$, we have
\begin{equation}
1 {\rm e}^{\dt D}G(t)={\rm e}^{\dt D}G(t) = G(t).
\label{fg}
\end{equation}
Trotter's formula then gives
\begin{eqnarray}
\exp[\dt(A(t)+D)]
=&&\lim_{n\rightarrow\infty}\Bigr( {\rm e}^{ {{\dt}\over n}A(t)}
     {\rm e}^{ {{\dt}\over n}D}\Bigr)^n,\nonumber\\
=&&\lim_{n\rightarrow\infty}
{\rm e}^{ {{\dt}\over n}A(t+\dt)}
\cdots
{\rm e}^{ {{\dt}\over n}A(t+{{2\dt}\over n})}
{\rm e}^{ {{\dt}\over n}A(t+{{\dt}\over n})},
\label{altas}
\end{eqnarray}
where property (\ref{fg}) has been applied repeatedly and accumulatively.
Comparing (\ref{tor2}) with (\ref{altas}) yields Suzuki's decomposition of
the time-ordered exponential\cite{suzu93}
\begin{equation}
\T\Bigl(\exp\int_t^{t+\dt}A(s)ds\Bigr)=\exp[\dt(A(t)+D)].
\label{tdecom}
\end{equation}
Thus time-ordering can be achieve by splitting an additional
operator $D$. This is extremely useful and transforms any existing splitting 
algorithms into integrators of non-autonomous equations. 
For example, one has the following symmetric splitting 
\be
\T_2(\ep)=\e^{\frac12 \ep D}\e^{\ep A(t)}\e^{\frac12 \ep D}=\e^{\ep A(t+\frac12 \ep)},
\la{sec}              
\ee
which is the second-order mid-point approximation.
Every occurrence of the operator $\e^{d_i\ep D}$, from right to left, updates the current 
time $t$ to $t+d_i\ep$. If $t$ is the time at the start of the
algorithm, then after the first occurrence of $\e^{\frac12 \ep D}$, time is $t+\frac12\dt$.
After the second $\e^{\frac12 \ep D}$, time is $t+\dt$. Thus the leftmost 
$\e^{\frac12 \ep D}$ is not without effect, it correctly updates the time for the next
iteration. Thus the iterations of $\T_2(\ep)$ implicitly imply
\ba
 \T_2^2(\ep/2)&=&\e^{\frac12\ep A(t+\frac34 \ep)}\e^{\frac12\ep A(t+\frac14 \ep)}\nn\\
 \T_2^3(\ep/3)&=&\e^{\frac13\ep A(t+\frac56 \ep)}\e^{\frac13\ep A(t+\frac12 \ep)}
 \e^{\frac13\ep A(t+\frac16 \ep)}\nn\\
&&\qquad\qquad\qquad\cdots
\la{tn}
\ea
For the odd-order basis,  we have
\ba
\cu_1(h)&=&\e^{\dt D }\e^{\dt A(t)}=\e^{\dt A(t)}\nn\\
\cu_2(h)&=&\e^{\frac13\dt D }\e^{\frac23\dt A(t)}\e^{\frac23\dt D }\e^{\frac13\dt A(t)}
=\e^{\frac23\dt A(t+\frac23 h)}\e^{\frac13\dt A(t)}\nn\\
\cu_3(h)&=&\e^{\frac25\dt A(t+\frac45 h)}
                             \e^{\frac25\dt A(t+\frac25 h)}\e^{\frac15\dt A(t)}\nn\\
&&\qquad\qquad\qquad\cdots
\la{tubas}
\ea

\begin{remark}
The recent work by Wiebe {\it et al.}\cite{wiebe08} suggests that Suzuki's decomposition
(\ref{tdecom}) only holds if $A(t)$ is sufficiently smooth. In cases where the
derivatives of $A(t)$ cease to exist, high-order integrators based on (\ref{tdecom})
maybe degraded to lower orders.  
\end{remark}

For $A(t)=T+V(t)$, since $[D,T]=0$, the second-order algorithm can be
obtained as
\ba
\label{error_2}
\T_2(\ep)
&=&{\rm e}^{{1\over 2}\dt (T+D) }\e^{\ep V(t)}{\rm e}^{{1\over 2}\dt (T+D) }\nn\\
&=&{\rm e}^{{1\over 2}\dt T }{\rm e}^{{1\over 2}\dt D }
\e^{\ep V(t)}{\rm e}^{{1\over 2}\dt D }{\rm e}^{{1\over 2}\dt T }\nn\\
&=&{\rm e}^{{1\over 2}\dt T }
{\rm e}^{\dt V(t+\dt/2)}
{\rm e}^{{1\over 2}\dt T }.
\ea
For odd order algorithms, we now have the following sequence of basis product
\ba
\cu_1(h)&=&\e^{\dt T }\e^{\dt V(t)}\nn\\
\cu_2(h)&=&\e^{\frac13\dt T }\e^{\frac23\dt V(t+\frac23 h)}\e^{\frac23\dt T }\e^{\frac13\dt V(t)}\nn\\
\cu_3(h)&=&\e^{\frac15\dt T }\e^{\frac25\dt V(t+\frac45 h)}\e^{\frac25\dt T }
                             \e^{\frac25\dt V(t+\frac25 h)}\e^{\frac25\dt T }\e^{\frac15\dt V(t)}\nn\\
&&\qquad\qquad\qquad\cdots
\la{tvubas}
\ea
While any power of $\T_2(\ep)$ is time-symmertic,
each $\cu_{n}(h)$ is {\it time asymmetric},
\be
\cu_{n}(-h)\cu_n(h)\ne 1.
\ee

\section{Errors and convergence of the Multi-product expansion}
\label{error}

While extrapolation methods are well-known in the study of
differential equations, there is a virtually no work done in the 
context of operators. Here, we extend the method of extrapolation
to the decomposition of two operators, which is the basis of the MPE method.
Working at the operator, rather than at the solution level, allows the
extrapolation method be widely applied to many
time-dependent equations. In particular, we will use the constructive 
details in\cite{chin08} to prove convergence results for the multi-product
expansion. While this work is restricted to exponential splitting, our
proof of convergence based on the general framework of \cite{han08}.

\subsection{Analysis of the even-order kernel $\T_2$}

We will assume that at sufficient small $h$, the Strang splitting is
bounded as follow:
\be
\label{equation_1}
 || \T_2(h) || = || \exp(\frac{1}{2} h D) \exp(h A(t)) \exp(\frac{1}{2} h D) || \le \exp(c \omega h),
\ee
with $c$ only depend on the coefficients of the method, see the work of convergence analysis 
on this splitting by Janke and Lubich \cite{jan00}.
We can then derive the following convergence results for the multi-product expansion.

\begin{theorem}
\label{conv_1}

For the numerical solution of (\ref{funeq}), we consider the
MPE algorithm (\ref{mulexp}) of order $2n$. Further we assume 
the error estimate in equation (\ref{equation_1}),
then we have the following convergence result:

\be
|| \left( S^m - \exp(m h (A(t) + D)) \right)u_0 || \le C O(h^{2n+1}) ,   m h \le t_{end} ,
\ee
where $ S = \sum_{i=1}^n c_i \T_2^{k_i}(\frac{h}{k_i})$ and $C$ is to be chosen uniformly on bounded time intervals and
independent of $m$ and $h$ for sufficient small $h$.

\end{theorem}

\begin{proof}

We apply the telescopic identity and obtain:
\ba
&& \left( S^m - \exp(m h (A(t) + D)) \right) u_0 = \\
&& \sum_{\nu = 0}^{m-1} S^{m - \nu - 1} ( S - \exp(h (A(t) + D)) ) \exp(\nu h (A(t) + D)) u_0.
\ea
where $ S = \sum_{i=1}^n c_i \T_2^{k_i}(\frac{h}{k_i})$

We apply the error estimate in (\ref{equation_1}) to obtain the stability requirement:

\be
|| \sum_{i=1}^n c_i \T_2^{k_i}(\frac{h}{k_i}) || \le \exp(c \omega h).
\ee

Assuming the consistency of
\be
|| \sum_{i=1}^n c_i \T_2^{k_i}(\frac{h}{k_i}) - \exp(h (A + D)) || \le C O(h^{2n+1})
\ee
we have the following error bound:
\be
|| \left( S^m - \exp(m h (A(t) + D)) \right)u_0 || \le C O(h^{2n+1}) ,   m h \le t_{end} ,
\ee

The consistency of the error bound is derived in the following theorem.

\end{proof}

\begin{theorem}
\label{conv_2}

For the numerical solution of (\ref{funeq}),
we have the following consistency:
\ba
&& || \sum_{i=1}^n c_i \T_2^{k_i}(\frac{h}{k_i}) - \exp(h (A + D)) || \le C O(h^{2n+1}) .   
\ea 
\end{theorem}

\begin{proof}

Based on the derivation of the coefficients via the Vandermonde 
equation the product is bounded and we have: 
\ba
&&  \sum_{k = 1}^n c_k \T_2^k(\frac{h}{k}) =   \sum_{k = 1}^n c_k \left( \exp((A + D) h) - (k^{-2} h^3 E_3 + k^{-4} h^5 E_5 + \ldots ) \right) , \\
&& =   \sum_{k = 1}^n c_k \left( \exp((A + D) h) - \sum_{i=1}^n k^{- 2i} h^{2i + 1} E_{2 i +1} \right) , \nonumber \\
&& = \left(  \exp((A + D) h) - \sum_{k = 1}^n c_k \sum_{i=1}^n k^{- 2i} h^{2i + 1} E_{2 i +1} \right) , \nonumber \\
&& = O(h^{2n+1}),
\ea 
where the coefficients are given in (\ref{coef}).

\end{proof}

\begin{lemma}
\label{lemma_1}
We assume $|| A(t) ||$ to be bounded in the interval $t \in (0, t_{end})$.
Then $\T_2$ is non-singular for sufficient small $h$.
\end{lemma}

\begin{proof}
We use our assumption $||A(t)||$ is to be bounded in the interval $0 < t < t_{end}$.

So we can find $||A(t)|| < C$ for $0<t<t_{end}$, where $C \in \R^+$ a bound of operator $A(t)$ independent of $t$.

Therefore $\T_2$ is always non-singular for sufficiently
small h. 

\end{proof}

\begin{remark}
Based on these results the kernel $\T_2$ is also uniform convergent.

The same argument can be used by applying to MPE formula,
while all operators are convergent, the sum of all is also bounded 
and convergent, see \cite{dv50} and \cite{egg50}.
\end{remark}

\begin{remark}
For higher kernels, e.g. 4th order, there exists also error bounds
so that uniform convergent results can be derived, see e.g. \cite{geiser08}.
Such kernels can also be used to the MPE method to acchieve
higher order accuracy with uniform convergent series. But as we
noted earlier, these cannot be applied to time-irreversible problems.
\end{remark}



\subsection{Analysis of the odd-order kernel ${\cal U}_n$}

\begin{lemma}
\label{lem_2}

We will assume that for sufficiently small $h$, 
the Burstein and Mirin's decomposition is bounded as follow:
\ba
\label{equ_kern_1}
 || {\cal U}_n(h) || = || \e^{\frac{h}{2n-1}A(t)}(\e^{\frac{2h}{2n-1}D}\e^{\frac{2h}{2n-1}A(t)})^{n-1}\e^{\frac{h}{2n-1}D} || \le  \exp(c \omega h), \; \forall t \ge 0 ,
\ea
with $c$ only dependent on the coefficients of the method.

\end{lemma}

The proof follows by rewriting equation (\ref{equ_kern_1})
as a product of the Strang and the A-B splitting schemes:

\begin{proof}

Equation (\ref{equ_kern_1}) can be rewritten as:
\ba
\label{ass1_1}
&& \e^{\frac{h}{2n-1}A(t)}(\e^{\frac{2h}{2n-1}D}\e^{\frac{2h}{2n-1}A(t)})^{n-1}\e^{\frac{h}{2n-1}D} \\
& = & \left( \e^{\frac{h}{2n-1}A(t)} \e^{\frac{2h}{2n-1}D} \e^{\frac{h}{2n-1}A(t)} \right)^{n-1} 
\e^{\frac{h}{2n-1}A(t)} \e^{\frac{h}{2n-1}D}, \; \forall t \ge 0 , \nonumber 
\ea
The error bound and underlying convergence analysis for both the Strang and the A-B
splitting have been previously studied by
Janke and Lubich \cite{jan00}.
\end{proof}

We assume the following derivation of the higher order MPE:

\begin{assumption}
\label{ass_2}

We assume the following higher order decomposition,
\be
{\rm e}^{\ep(A + D)}
=\sum_{i=1}^n \tilde{c}_i \; {\cal U}_i(h) + O(h^{2n}).
\la{mulexp_2}
\ee
where $\tilde{c}_i$ are derived
based on the Vandermonde equation (\ref{vand}) with $\{k_i\}$ being
a set of odd whole numbers.

\end{assumption}

We can then derive the following convergence results for the multi-product expansion.

\begin{theorem}
\label{conv_3}

For the numerical solution of (\ref{funeq}), we consider the Assumption \ref{ass_2} of order $2n - 1$ and we apply Lemma \ref{lem_2},
then we have a convergence result given as:

\be
|| \left( S^m - \exp(m h (A(t) + D)) \right)u_0 || \le C O(h^{2n}) ,   m h \le t_{end} ,
\ee
with $n = 1, 2, 3, \ldots$, and where $ S = \sum_{i=1}^n \tilde{c}_i \U_i(h)$ and $C$ is to be chosen uniformly on bounded time intervals and
independent of $m$ and $h$ for sufficient small $h$.

\end{theorem}

\begin{proof}

The same proof ideas can be followed after the proof of Theorem \ref{conv_1}.

The consistency of the error bound is derived in the following theorem.

\end{proof}

\begin{theorem}
For the numerical solution of (\ref{funeq}),
we have the following consistency:
\ba
&& || \sum_{i=1}^n \tilde{c}_i \; \U_i(h) - \exp(h (A + D)) || \le C O(h^{2n}) .   
\ea 
\end{theorem}

\begin{proof}

The same proof ideas can be followed after the proof of Theorem \ref{conv_2}.

\end{proof}

\begin{remark}
The same proof idea can be used to generalise the 
higher order schemes.
\end{remark}

\section{Analytical and numerical verifications}
\label{exp}

In this section, we seek to verify and assess the convergence of both 
the even and odd order MPE algorithms. For a single product splitting, there are
no known splittings that are exact in the limit of large number of
operators. Even in the case of the Zassenhaus formula, it is 
non-trivial to compute the higher order products, not to mention
evaluating them. For this purpose, we turn to the much studied
Magnus expansion, where the exact limit can be computed in simple
cases.

The Magnus expansion\cite{blan08} solves (\ref{funeq}) in the form
\begin{eqnarray}
\label{magform}
&& Y(t) = \exp(\Omega(t))Y(0),\qquad \Omega(t) = \sum_{n=1}^{\infty} \Omega_n(t) 
\end{eqnarray}
where the first few terms are
\begin{eqnarray}
\label{equ3}
&& \Omega_1(t) = \int_0^tdt_1 A_1 \nn\\
&& \Omega_2(t) = \frac12 \int_0^t dt_1\int_0^{t_1}dt_2 [A_1, A_2]   \nn\\
&& \Omega_3(t) = \frac16 \int_0^t dt_1\int_0^{t_1}dt_2
\int_0^{t_2}dt_3 ( [A_1, [A_2,A_3]+[[A_1,A_2],A_3] )\nn\\
&&\qquad\qquad\cdots\cdots
\la{magterms}
\end{eqnarray}
with $A_n\equiv A(t_n)$. In practice, it is more useful to define the $n$th order
Magnus operator 
\be
\Omega^{[n]}(t)=\Omega(t)+O(t^{n+1})
\ee
such that
\be
Y(t)=\exp\bigl[\Omega^{[n]}(t)\bigr]Y(0)+O(t^{n+1}).
\la{magnexp}
\ee 
Thus the second-order Magnus operator is
\ba
 \Omega^{[2]}(t)&=&\int_0^tdt_1 A(t_1)\nn\\
 &=&tA\left(\frac12 t\right)+O(t^3) 
\ea
and a fourth-order Magnus operator\cite{blan08} is
\be
 \Omega^{[4]}(t)=\frac12 t(A_1+A_2)-c_3 t^2[A_1,A_2]
\ee
where $A_1=A(c_1t)$, $A_2=A(c_2 t)$ and
\be
c_1 = \frac12 - \frac{\sqrt{3}}6,\qquad c_2 = \frac12 + \frac{\sqrt{3}}6, 
\qquad c_3=\frac{\sqrt{3}}{12}.
\ee
For the ubiquitous case of 
\be
A(t)=T+V(t),
\ee
one has
\ba
\e^{\Omega^{[2]}(t)}&=&\e^{t[T+V(t/2)]}\nn\\
&=& \e^{\frac12 tT}\e^{tV(t/2)}\e^{\frac12 t T}+O(t^3)
\ea
and
\be
\e^{\Omega^{[4]}(t)}=\e^{c_3 t (V_2-V_1)}
\e^{t (T+\frac12(V_1+V_2))}
\e^{-c_3t (V_2-V_1)}+O(t^5)
\la{amag4}              
\ee
where 
\be
V_1 = V(c_1 t),\qquad V_2 = V(c_2 t).
\ee
The Magnus expansion (\ref{magnexp}) is automatically sturcture-preserving because it is 
a {\it single exponential operator} approximation. However, since one must further split 
$\Omega^{[n]}$ into computable parts, the expansion is as complex, if not more so, than 
a single product splitting. In the following, the comparison is not strictly equitable,
because the MPE is {\it not} structure-preserving. Neveetheless it is useful to know
that, perhaps for that reason, MPE can be uniformly convergent.      

\subsection{The non-singular matrix case}

To assess the convergence of the Multi-product expansion
with that of  the Magnus series, consider the well known
example\cite{moan08} of
\begin{eqnarray}
&& A(t) = \left(
\begin{array}{c c}
2 & t \\
0 & -1
\end{array}
\right).
\end{eqnarray}
The exact solution to (\ref{funeq}) with $Y(0)=I$ is
\ba
&& Y(t) = \left(
\begin{array}{c c}
\e^{2 t} & \quad f(t) \\
0 &  \e^{- t}
\end{array}
\right),
\la{exysol}
\ea
with
\ba
f(t)&=&\frac{1}{9}\e^{-t}(\e^{3t}-1-3t)\la{exff}\\
&=&	\frac{t^2}2
+\frac{t^4}8
+\frac{t^5}{60}
+\frac{t^6}{80}
+\frac{t^7}{420}
+\frac{31 t^8}{40320}
+\frac{t^9}{6720}
+\frac{13t^{10}}{403200}
+\frac{13t^{11}}{178200}
\la{exffb}
\ea
For the Magnus expansion, one has the series
\be
\Omega(t) = \left(
\begin{array}{c c}
2t & \quad g(t) \\
0 & -t
\end{array}
\right),
\la{omga}
\ee
with, up to the 10th order,
\ba
g(t)&=&\frac12 t^2-\frac14 t^3+\frac3{80}t^5-\frac9{1120}t^7+\frac{81}{44800}t^9+\cdots\la{gser}\\
&\rightarrow& \frac{t(\e^{3t}-1-3t)}{3(\e^{3t}-1)}.
\ea 
Exponentiating (\ref{omga}) yields
(\ref{exysol}) with
\ba
f(t)&=&t\e^{-t}(\e^{3t}-1)\left(\frac16 -\frac1{12}t+\frac1{80}t^3-\frac3{1120} t^5
+\frac{27}{44800}t^7+\cdots\right)\la{fser}\\
&\rightarrow& t\e^{-t}(\e^{3t}-1)\left(\frac1{9t}-\frac1{3(\e^{3t}-1)}\right)
\ea
Whereas the exact solution (\ref{exff}) is an entire function of $t$, the
Magnus series (\ref{gser}) and (\ref{fser}) only converge for $|t|<\frac23\pi$
due to the pole at $t=\frac23\pi i$. The Magnus series (\ref{fser}) is plot in
Fig.1 as blue lines. The pole at $|t|=\frac23\pi\approx 2$ is clearly visible.

\begin{figure}[ht]
\begin{center} 
\includegraphics[width=10.0cm,angle=-0]{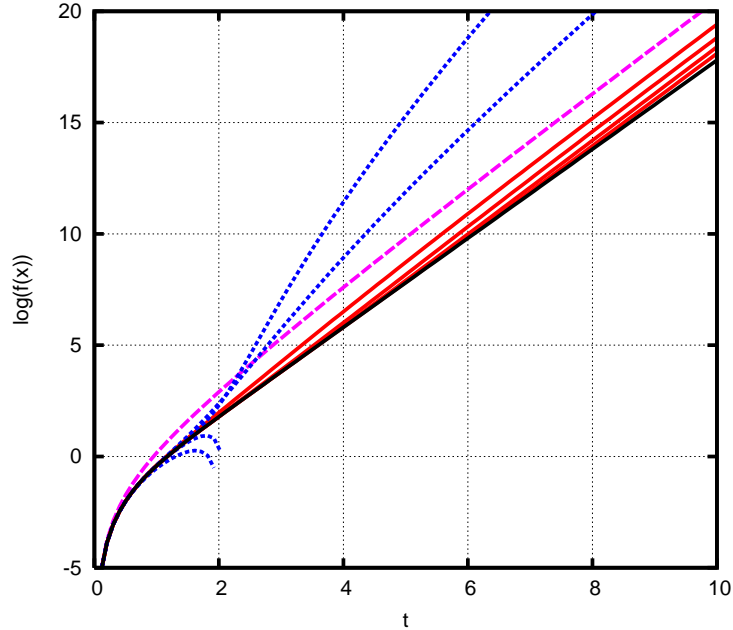} 
\end{center}
\caption{\label{fig_1} The black line is the exact result (\ref{exff}).
The dotted blue lines are the Magnus fourth to tenth order results (\ref{fser}),
which diverge from the exact result beyond $t>2$. The solid red lines are
the multi-product expansions. The dashed-purple line is their common second order result.
}
\end{figure}

For the even order multi-product expansion,
from (\ref{sec}), by setting $\dt=t$ and $t=0$, we have
\be
{\cal T}_2(t)=\exp\left[ t 
\left(
\begin{array}{c c}
2 & \quad\frac12 t \\
0 & -1
\end{array}
\right)\right]
=\left(
\begin{array}{c c}
\e^{2 t} & \quad f_2(t) \\
0 &  \e^{- t}
\end{array}
\right)	
\ee
and we compute ${\cal T}_2^2(t)$ according to (\ref{tn}) as
\be
{\cal T}^2_2(t/2)=\exp\left[ \frac{t}{2} 
\left(
\begin{array}{c c}
2 & \quad\frac34 t \\
0 & -1
\end{array}
\right)\right]
\left(
\begin{array}{c c}
\e^{t} & \quad f_2(t/2) \\
0 &  \e^{- \frac{t}{2}}
\end{array}
\right)	
\ee

with
\be
f_2(t) =\frac16 t\e^{-t}(\e^{3 t}-1). 
\ee
This is identical to first term of the Magnus series (\ref{fser}) and is
an entire function of $t$. Since higher order MPE uses only powers of
${\cal T}_2$, higher order MPE approximations are also entire functions of
$t$. Thus up to the 10th order, one finds
\be
f_4(t) =t\e^{-t}\left(\frac{\e^{3 t}-5}{18}+\frac{2\e^{3t/2}}9\right) 
\ee
\be
f_6(t) = t\e^{-t}\left(\frac{11\e^{3 t}-109}{360}+\frac9{40}(\e^{2 t}+\e^t)
-\frac8{45}\e^{3t/2}\right) 
\ee
\be
f_8(t) = t\e^{-t}\left(\frac{151\e^{3 t}-2369}{7560}+\frac{256}{945}(\e^{9t/4}+\e^{3t/4})
-\frac{81}{280}(\e^{2 t}+\e^t)
+\frac{104}{315}\e^{3t/2}\right) 
\ee
\ba
f_{10}(t) &=& t\e^{-t}\biggl(
  \frac{15619\e^{3 t}-347261}{1088640}
+\frac{78125}{217728}(\e^{12t/5}+\e^{9t/5}+\e^{6t/5}+\e^{3t/5})\nn\\
&&\qquad\qquad -\frac{4096}{8505}(\e^{9t/4}+\e^{3t/4})
+\frac{729}{4480}(\e^{2 t}+\e^t)
-\frac{4192}{8505}\e^{3t/2}\biggr). 
\ea
These even order approximations are plotted as red lines in Fig.1. The
convergence is uniform for all $t$.

When expanded, the above yields
\ba
\label{expan_1}
f_{2}(t)&=&	\frac{t^2}2+\frac{t^3}4+\cdots\nn\\
f_{4}(t)&=&	\frac{t^2}2+\frac{t^4}8+\frac{5t^5}{192}+\cdots\nn\\
f_{6}(t)&=&	\frac{t^2}2+\frac{t^4}8+\frac{t^5}{60}+\frac{t^6}{80}+\frac{t^7}{384}+\cdots\nn\\
f_{8}(t)&=&	\frac{t^2}2+\frac{t^4}8+\frac{t^5}{60}+\frac{t^6}{80}
+\frac{t^7}{420}+\frac{31t^8}{40320}+\frac{1307t^9}{8601600}
+\cdots\nn\\
f_{10}(t)&=&	\frac{t^2}2+\frac{t^4}8+\frac{t^5}{60}+\frac{t^6}{80}
+\frac{t^7}{420}+\frac{31t^8}{40320}
+\frac{t^9}{6720}
+\frac{13t^{10}}{403200}
+\frac{13099t^{11}}{232243200}\qquad
\la{evenorder}
\ea
and agree with the exact solution to the claimed order. Similarly, the m-step
extrapolated algorithms $\T_{2,m}$, $\T_{4,m}$, {\it etc.}, are also correct up to
the claimed order.

For odd orders, by again setting $\dt=t$ and $t=0$, 
the basis defined in (\ref{tubas}) now reads
\ba
\cu_1(t)&=&\exp\left[ t 
\left(
\begin{array}{c c}
2 & \quad 0 \\
0 & -1
\end{array}
\right)\right]
=\left(
\begin{array}{c c}
\e^{2 t} & \quad 0 \\
0 &  \e^{- t}
\end{array}
\right)\nn\\	
\cu_2(t)&=&
\exp\left[ \frac23 t 
\left(
\begin{array}{c c}
2 & \quad \frac23 t \\
0 & -1
\end{array}
\right)\right]
\exp\left[ \frac13 t 
\left(
\begin{array}{c c}
2 & \quad 0 \\
0 & -1
\end{array}
\right)\right]\nn\\
&=&\left(
\begin{array}{c c}
\e^{2 t} & \quad\frac29 t(\e^t-e^{-t}) \\
0 &  \e^{- t}
\end{array}
\right)\nn\\
\cdots	
\ea
and the MPE (\ref{three}) to (\ref{nine}) give
\ba
\label{expan_2}
f_{3}(t)&=&	\frac{t^2}2+\frac{t^4}{12}+\cdots\nn\\
f_{5}(t)&=&	\frac{t^2}2+\frac{t^4}8+\frac{t^5}{60}+\frac{11 t^6}{1000}+\cdots\nn\\
f_{7}(t)&=&	\frac{t^2}2+\frac{t^4}8+\frac{t^5}{60}+\frac{t^6}{80}+\frac{t^7}{420}
+\frac{18299t^8}{24696000}\cdots\nn\\
f_{9}(t)&=&	\frac{t^2}2+\frac{t^4}8+\frac{t^5}{60}+\frac{t^6}{80}
+\frac{t^7}{420}+\frac{31t^8}{40320}
+\frac{t^9}{6720}+\frac{1577t^{10}}{49392000}
+\cdots
\la{oddorder}
\ea
Results (\ref{evenorder}) and (\ref{oddorder}) constitute an analytical verification of
the even and odd order MPE (\ref{four})-(\ref{ten}) and (\ref{three})-(\ref{nine}). 

\subsection{The singular matrix case}

Consider the radial Schr\"odinger equation\\
\begin{eqnarray}
  \frac{\partial^2 u}{\partial r^2} = f(r, E) u(r)
  \la{har}
\end{eqnarray}
where
\begin{eqnarray}
  f(r, E) = 2 V(r) - 2 E + \frac{l(l + 1)}{r^2} \; .
\end{eqnarray}
By relabeling $r\rightarrow t$ and $u(r)\rightarrow q(t)$, (\ref{har}) can be viewed as
harmonic oscillator with a time dependent spring constant
\be
k(t,E)=-f(t,E)
\ee
and  Hamiltonian
\begin{eqnarray}
H=\frac12 p^2+\frac12 k(t,E)q^2.
\la{tham}
\end{eqnarray}
Thus any eigenfunction of (\ref{har}) is an exact time-dependent solution of
(\ref{tham}). For example, the ground state of the hydrogen atom with
$l=0$, $E=-1/2$ and
\be
V(r)=-\frac1{r}
\ee
yields the exact solution
\ba
q(t)&=&t\exp(-t)\nn\\
&=&t-t^2+\frac{t^3}{2}-\frac{t^4}{6}+\frac{t^5}{24}-\frac{t^6}{120}
+\frac{t^7}{720} -\frac{t^8}{5040}
\cdots,\nn\\
&=&t-t^2+\frac{t^3}{2}-0.1667t^4+0.0417t^5-0.0083t^6
\cdots
\la{exsolu}
\ea
with initial values $q(0)=0$ and $p(0)=1$.
Denoting
\be
Y(t) = \left(
\begin{array}{c}
q(t) \\
p(t)
\end{array}
\right), 
\ee
the time-dependent harmonic oscillator (\ref{tham}) now corresponds to
\begin{eqnarray}
&& A(t) = \left(
\begin{array}{c c}
0 & 1 \\
f(t) & 0
\end{array}
\right)=
\left(
\begin{array}{c c}
0 & 1 \\
0 & 0
\end{array}
\right)
+
\left(
\begin{array}{c c}
0 & 0 \\
f(t) & 0
\end{array}
\right) \equiv T+V(t),
\la{coula}
\end{eqnarray}
with a singular matrix element
\be
f(t)=(1-\frac2{t}).
\ee
The second-order midpoint algorithm is
\ba
\T_2(h,t)&=&{\rm e}^{{1\over 2}h T }
{\rm e}^{h V(t+h/2)}
{\rm e}^{{1\over 2}h T }\nn\\
&=&\left(
\begin{array}{c c}
1+\frac12 h^2f(t+\frac12 h)& \quad h +\frac14 h^3f(t+\frac12 h) \\
h f(t+\frac12 h) & 1+\frac12 h^2f(t+\frac12 h),
\end{array}
\la{sectt}
\right)              
\ea
and for $q(0)=0$ and $p(0)=1$, (setting $t=0$ and $h=t$),
correctly gives the second order result,
\be
q_2(t)=t +\frac14 t^3f(\frac12 t)= t-t^2+\frac14 t^3.
\ee
The even order multi-product expansions (\ref{four})-(\ref{ten}) then yield
\ba
q_4(t)&=&t-t^2+0.3889 t^3-0.1111 t^4+0.0104 t^5\nn\\
q_6(t)&=&t-t^2+0.4689 t^3-0.1378 t^4+0.0283 t^5-0.0043 t^6 \nn\\
q_8(t)&=&t-t^2+0.4873 t^3-0.1542 t^4+0.0356 t^5-0.0062 t^6\cdots\nn\\
q_{10}(t)&=&t-t^2+0.4936t^3-0.1603 t^4+0.0385 t^5-0.0073 t^6\cdots
\la{evorder}
\ea
where we have converted fractions to decimal forms for easier comparison with 
the exact solution (\ref{exsolu}). 
One sees that MPE no longer matches the Taylor expansion 
beyond second-order. This is due to the singular nature of the Coulomb potential,
which makes the problem a challenge to solve. (If one naively makes a Taylor expansion
about $t=0$ starting with $q(0)=0$ and $p(0)=1$, then every term beyond the initial
values would either be divergent or undefine.) 

Since $A(t)$ is now singular
at $t=0$, the previous proof of uniform convergence no longer holds. Nevertheless, from
the exact solution (\ref{exsolu}), one sees that force (or acceleration)
\be
\lim_{t\rightarrow 0} f(t)q(t)=-2
\la{fin}
\ee
remains finite. It seems that this is sufficient for uniform convergence
as the coefficients of $t^n$  
do approach their correct value with increasing order. 

For odd order MPE, 
while each term $\e^{(\dt/x) V(t)}$ of the basis product in 
(\ref{tvubas}) is singular at $t=0$, but because of (\ref{fin}),
\be
 \lim_{t\rightarrow 0}\e^{(\dt/x) V(t)}
 \left(
\begin{array}{c}
q(t) \\
p(t)
\end{array}
\right)=
\left( 
\begin{array}{c}
0 \\
1-2\dt/x
\end{array}
\right). 
\la{sing}
\ee
Interpreting the action of the first operator this way, the basis products of (\ref{tvubas})
then yield, according to the MPE (\ref{three})-(\ref{nine}),
\ba
q_3(t)&=&t-t^2+\frac{t^3}{2}-0.1111 t^4\nn\\
q_5(t)&=&t-t^2+\frac{t^3}{2}-0.1458 t^4+0.0333 t^5-0.0033 t^6 \nn\\
q_7(t)&=&t-t^2+\frac{t^3}{2}-0.1628 t^4+0.0382 t^5-0.0067 t^6\cdots\nn\\
q_9(t)&=&t-t^2+\frac{t^3}{2}-0.1655 t^4+0.0406 t^5-0.0078 t^6\cdots
\la{odorder}
\ea
Now $q_3(t)$ is correct to third order, but higher order algorithms are
still down-graded and only approaches the exact solution asymptotically but 
uniformly.

\begin{figure}[ht]
\begin{center} 
\includegraphics[width=10.0cm,angle=-0]{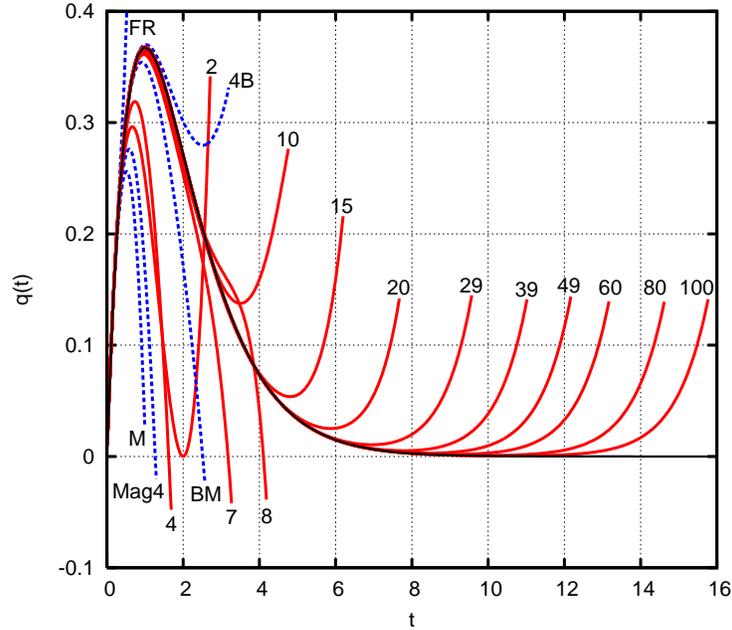} 
\end{center}
\caption{\label{part_1} The uniform convergence of the multi-product expansion 
in solving for the hydrogen ground state wave function.
The black line is the exact ground state wave function. The numbers denote
the order of the multi-product expansion. The dotted blue lines denote results of
various fourth-order algorithms. }
\la{hydro}
\end{figure}

To see this uniform 
convergence, we show in Fig.\ref{hydro}, how higher order 
MPE, both even and odd, up to the 100th order, compares with the exact solution.
The calculation is done numerically rather than by evaluating the analytical
expressions such as (\ref{evorder}) or (\ref{odorder}). The order of the MPE algorithms are
indicated by numbers. For odd order algorithms, we do not even bother to
incorporate (\ref{sing}), but just avoid the singularity by starting the
algorithm at $t=10^{-6}$. 
Also shown are some
well know fourth-order symplectic algorithm FR (Forest-Ruth\cite{fr90}, 3 force-evaluations),
M (McLachlan\cite{mcl95}, 4 force-evaluations), BM (Blanes-Moan\cite{bm02}, 6 force-evaluations), 
Mag4 (Magnus integrator, 4 force-evaluations) and
4B\cite{chin06} (a {\it forward} symplectic algorithm with $\approx$ 2 evaluations).
These symplectic integrators steadily improves from FR, to M, to Mag4, to BM to
4B. Forward algorithm 4B is noteworthy in that it is the only fourth-order algorithm
that can go around the wave function maximum at $t=1$, yielding
\be
q_{4B}(t)=t-t^2+\frac{t^3}{2}-0.1635t^4+0.0397t^5-0.0070 t^6
+0.0009t^7
\cdots,
\ee
with the correct third-order coefficient and comparable higher order coefficients as
the exact solution (\ref{exsolu}). By contrast,
the FR algorithm, which is well know to have rather large errors, has the
expansion,   
\be
q_{FR}(t)=t-t^2-0.1942{t^3}+3.528t^4-2.415t^5+0.5742 t^6
-0.0437t^7
\cdots,
\ee
with terms of the wrong signs beyond $t^2$.	The failure of these fourth-order algorithms
to converge correctly due to the singular nature of the Coulomb potential is consistent
with the findings of Wiebe {\it et al.}\cite{wiebe08}. However, their finding does not
explain why the second-order algorithm can converge correctly and only higher order
algorithms fail. A deeper understanding of Suzuki's method is necessary to resolve this
issue. 

\begin{figure}[ht]
\begin{center} 
\includegraphics[width=10.0cm,angle=-0]{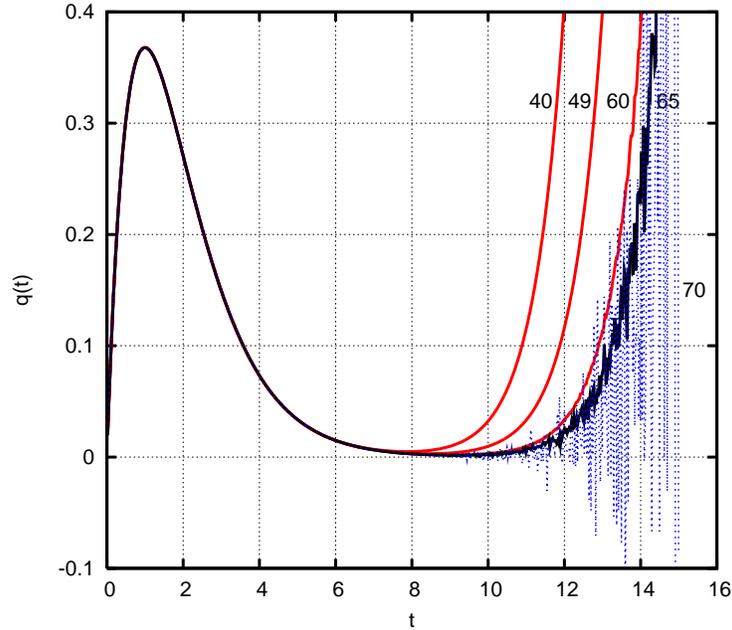} 
\end{center}
\caption{\label{part_1} The multi-product expansion results for 
the hydrogen ground state wave function using only double-precision
arithematics. The effect of limited precision begins to show at 
order 60. The results for order 65 and 70 are indicated as black and dotted blue lines
respectively.}
\la{noise}
\end{figure}

In Fig.\ref{hydro}, results for orders 60, 80 and 100 are computed using 
quadruple precision. If one uses only double-precision, the effect
of round-off errors on limited precision is as shown in Fig.\ref{noise}. 
For this calculation, the round-off errors are not very noticeable
even at orders as high as 40 or 49, which is rather surprising.
The round-off errors are noticeable only at orders greater than $\approx$ 60.  

For non-singular potentials such as the radial harmonic oscillator with
\be
f(t)=t^2-3,
\ee
and exact ground state solution
\be
q(t)=t\e^{-t^2/2}=t-\frac{t^3}{2}+\frac{t^5}{8}-\frac{t^7}{48}+\frac{t^9}{384}
-\frac{t^{11}}{3840}+\cdots, 
\ee
the multi-product expansion has no problem in reproducing
the exact solution to the claimed order:
\ba
q_6(t)&=&t-\frac{t^3}{2}+\frac{t^5}{8}-\frac{13 t^7}{576}+\cdots\nn\\
q_7(t)&=&t-\frac{t^3}{2}+\frac{t^5}{8}-\frac{t^7}{48}+\frac{1082t^9}{385875}+\cdots\nn\\
q_8(t)&=&t-\frac{t^3}{2}+\frac{t^5}{8}-\frac{t^7}{48}+\frac{20803t^9}{7741440}+\cdots\nn\\
q_9(t)&=&t-\frac{t^3}{2}+\frac{t^5}{8}-\frac{t^7}{48}+\frac{t^9}{384}-\frac{341t^{11}}{1224720}+\cdots\nn\\
q_{10}(t)&=&t-\frac{t^3}{2}+\frac{t^5}{8}-\frac{t^7}{48}+\frac{t^9}{384}-\frac{50977t^{11}}{193536000}+\cdots\ .
\la{hoorder}
\ea
In this case, the odd order algorithms have the advantage of
being correct one order higher.

\section{Concluding summary and discussions }
\label{conc}

In this work, we have shown that the most general framework
for deriving Nystr\"om type algorithms for solving autonomous and non-autonomous
equations is multi-product splitting. By expanding on a suitable
basis of operators, the resulting multi-product expansion 
not only can reproduce conventional extrapolated integrators of even-order
but can also yield new odd-order algorithms.   
By use of Suzuki's rule of incorporating the time-ordered exponential, any
multi-product splitting algorithm can be adopted for solving
explicitly time-dependent problems.	The analytically know expansion coefficients
$c_i$ allow great flexibility in designing adaptive algorithms.  
Unlike structure-preserving methods, such as the Magnus expansion,
which has a finite radius of convergence, our multi-product expansion converges 
uniformly. Moreover, MPE requires far less operators at higher orders than either 
the Magnus expansion or conventional single-product splittings. The general 
order-condition for multi-product splitting is not known and should be developed.   
In the future we will focus on applying MPE methods for solving
nonlinear differential equations and time-irreversible or semi-group problems.

\bibliographystyle{plain}

\begin{thebibliography}{10}


\bibitem{alb55}J. Albrecht,{\em Beitr\"age zum Runge-Kutta-Verfahren}, Zeitschrift f\"ur Angewandte
               Mathematik und Mechanik, 35(1955)100-110 , reproduced in Ref.\cite{bat99}.
\bibitem{bat99}R. H. Battin, {\it An Introduction to the
          Mathematics and Methods of Astrodynamics, Reviesed Edition}, AIAA, 1999.

\bibitem{blan99}S. Blanes, F. Casas, and J. Ros, Siam J. Sci. Comput., {\bf 21}, 711 (1999).



\bibitem{bm02}S. Blanes and P. C. Moan, 
               {\em Practical symplectic partition Runge-Kutta methods
			     and Runge-Kutta Nystr\"om methods}, 
                J. Comput. Appl. Math. {\bf 142}, 313 (2002).


\bibitem{blan08}
S.~Blanes, F.~Casas, J.A.~Oteo and J.~Ros.
\newblock{\em The Magnus expansion and some of its applications}.
\newblock arXiv.org:0810.5488 (2008). 

\bibitem{bcr99} S. Blanes, F. Casas and J. Ros,
                {\em Extrapolation of symplectic integrators}, 
				Celest. Mech. Dyn. Astron., {\bf 75} (1999)149-161

\bibitem{brank89} R. W. Brankin , I. Gladwell , J. R. Dormand , P. J. Prince , W. L. Seward, 
                  {\em Algorithm 670: a Runge-Kutta-Nyström code}, 
                  ACM Trans. Math. Softw. (TOMS), 15(1989)31-40 .

\bibitem{bur70} S. Z. Burstein and A. A. Mirin,
\newblock{\em Third order difference methods for hyperbolic equations}.
\newblock J. Comp. Phys. {\bf 5}, 547-571 (1970).

\bibitem{cm00} R. Chan and A. Murus,
                {\em Extrapolation of symplectic methods for Hamiltonian problems}, 
				Appl. Numer. Math., {\bf 34}(2000)189-205

\bibitem{chin97}S. A. Chin, {\em Symplectic Integrators From Composite 
                  Operator Factorizations}, Phys. Lett. {\bf A226}, 344 (1997).
\bibitem{chin02} 
S.A.~Chin and C.R.~Chen.
\newblock{\em Gradient symplectic algorithms for solving the Schr\"odinger equation with time-dependent potentials}.
\newblock Journal of Chemical Physics, 117(4), 1409-1415 (2002).


\bibitem{chin06} 
S.A.~Chin and P.~Anisimov.
\newblock{\em Gradient Symplectic Algorithms for Solving the Radial Schr\"odinger Equation}.
\newblock J. Chem. Phys. 124, 054106, 2006.

\bibitem{chin08} 
S.A.~Chin.
\newblock{\em Multi-product splitting and Runge-Kutta-Nystr\"om integrators},
Celest. Mech. Dyn. Astron. 106, 391-406 (2010).

\bibitem{chin093}S. A. Chin, S. Janecek and E. Krotscheck,
         {\em Any order imaginary time propagation method for solving the Schr\"odinger equation},
		 Chem. Phys. Lett.  470(2009)342-346.

\bibitem{dorm87} J. Dormand, M. El-Mikkawy and P. Prince,  
         {\em High-Order embedded Runge--Kutta--Nyström formulae}, 
	      IMA J. Numer. Analysis,  7(1987)423--430 . 

\bibitem{dv50}
A.~Dvoretzky and C.A.~Rogers,
\newblock{\em Absolute and Unconditional Convergence in Normed Linear Spaces.}
\newblock Proc. Natl. Acad. Sci. U S A, 36(4), 192-197, 1950.



\bibitem{dy76}
F.J.~Dyson.
\newblock {\em The radiation theorem of Tomonaga}.
\newblock Swinger and Feynman, Phys. Rev., 75, 486-502 (1976).




\bibitem{egg50}
H.G.~Eggleston,
\newblock{\em Some Remarks on Uniform Convergence.}
\newblock Proceedings of the Edinburgh Mathematical Society (Series 2), Cambridge University Press, 10: 43-52, 1953.  

\bibitem{fr90}E. Forest and R. D. Ruth, {\em 4th-order symplectic integration}.
               Physica D {\bf 43}, 105 (1990).

\bibitem{geiser08}
J.~Geiser.
\newblock {\em Fourth-order splitting methods for time-dependent
differential equation}
\newblock Numerical Mathematics: Theory, Methods and Applications, Global science press, Hong Kong, China, 1(3), 321-339, 2008.

\bibitem{gragg65} W. B. Gragg, {\em On extrapolation algorithms for ordinary initial
          value problems}, SIAM J. Number. Anal. {\bf 2}(1965)384-404.


\bibitem{han08}
E.~Hansen and A.~Ostermann.
\newblock {\em Exponential splitting for unbounded operators.}
\newblock Mathematics of Computation, accepted, 2008.

\bibitem{hai03}
E.~Hairer, C. Lubich, and G.~Wanner.
\newblock {\em Geometric Numerical Integration: Structure-Preserving
Algorithms for Ordinary Differential Equations.}
\newblock SCM, Springer-Verlag Berlin-Heidelberg-New York,
          No. 31, 2002.

 \bibitem{hair93}E. Hairer, S.P. Norsett and G. Wanner, {\it Solving Ordinary Differential 
                Equations I - Nonstiff Problems}, Second Edition, Springer-Verlag, Berlin, 1993. 				   				      

\bibitem{hil74} F. B. Hildebrand,
\newblock{\em Introduction to Numerical Analysis}, Second Edition,
\newblock McGraw-Hill, New York, 1974, Dover Edition, 1987, P.290.

\bibitem{hoch03}
M.~Hochbruck and Chr.~Lubich.
\newblock{\em On Magnus Integrators for Time-Dependent Schr\"odinger Equations}
\newblock SIAM Journal on Numerical Analysis, 41(3):  945-963, 2003.





\bibitem{jan00} 
T.~Jahnke and C.~Lubich.
\newblock {\em Error bounds for exponential operator splittings.}
\newblock BIT Numerical Mathematics, 40:4, 735-745, 2000.


\bibitem{mcl95}R. I. McLachlan, 
         {\em On the numerical integration of ordinary differential equations
		  by symmetric composition methods}, 
		   SIAM J. Sci. Comput. {\bf 16}, 151 (1995).

\bibitem{moan08}
P.C.~Moan and J.~Niesen
\newblock{\em Convergence of the Magnus series}.
\newblock J. Found. of Comp. Math., 8(3):291--301 (2008). 

\bibitem{neri87}F. Neri, ``Lie Algebra and Canonical Integration",
                Department of Physics, Univeristy of Maryland print,1987.

 \bibitem{nys25} E. J. Nystr\"om,{\em \"Uber die Numerische Integration von Differentialgleichungen},
                Acta Societatis Scientiarum Ferrica, 50(1925)1-55 .		  		  


\bibitem{schat94}M. Schatzman, 
         {\em Higher order alternate direction methods},
          Compt. Meth. Appl. Mech. Eng, 116(1994)219-225 .

\bibitem{sheng89}Q. Sheng, 
         {\em Solving linear partial differential equations by exponential splitting},
          IMA Journal of Numer. Analysis, 9(1989)199-212 .



 \bibitem{strang68} G. Strang,
 \newblock{\em On the construction and comparison of difference schemes}
 \newblock SIAM J. Numer. Anal. {\bf 5}, 506-517 (1968).

\bibitem{suzuki91}M. Suzuki, 
         {\em General theory of fractal path-integrals with applications to many-body theories and
		  statistical physics}, J. Math. Phys. {\bf 32}, 400 (1991).

\bibitem{suzu93}
M. Suzuki.
\newblock{\em General Decomposition Theory of Ordered Exponentials}.
\newblock Proc. Japan Acad., vol. 69, Ser. B, 161 (1993).

\bibitem{suzu96}M. Suzuki, M.: 1996,
         ``New scheme of hybrid Exponential product formulas with applications 
		  to quantum Monte Carlo simulations" in
		  {\it Computer Simulation Studies in Condensed
         Matter Physics VIII}, eds, D. Landau, K. Mon and H. Shuttler, 
		   Springler, Berlin, P.1-6, 1996.

\bibitem{wiebe08}
N.~Wiebe, D.W.~Berry, P.~Hoyer and B.C.~Sanders.
\newblock{\em Higher Order Decompositions of Ordered Operator Exponentials}.
\newblock arXiv.org:0812.0562 (2008).

\bibitem{yoshida80}
K.~Yoshida.
\newblock{Functional Analysis}.
\newblock Classics in Mathematics, Springer-Verlag, Berlin-Heidelberg-New
York (1980).

\bibitem{yoshida90}
H.~Yoshida.
\newblock{Construction of higher order symplectic integrators}
\newblock Physics Letters A, Vol. 150, no. 5,6,7, 1990.


\bibitem{zill10}R. E. Zillich, J. M. Mayrhofer, S. A. Chin,
        {\em Extrapolated high-order propagator for path integral Monte Carlo simulations},
		J. Chem. Phys. {\bf 132} 044103 (2010). 



\end{thebibliography}

\end{document}